\documentclass[leqno,12pt]{amsart}
\usepackage{amsmath,amssymb,amsthm,amsfonts,amscd}

\newtheorem{dummy}{anything}
\newtheorem{theorem}[dummy]{Theorem}
\newtheorem{lemma}[dummy]{Lemma}
\newtheorem{proposition}[dummy]{Proposition}
\newtheorem{corollary}[dummy]{Corollary}
\newtheorem{claim}{Claim}
\newtheorem{definition}{Definition}

\newcommand{\Hom}{{\rm {Hom}}}

\newcommand{\C}{{\mathbf C}}

\newcommand{\Z}{{\mathbf Z}}
\newcommand{\N}{{\mathbf N}}

\newcommand{\tc}{{\sf {TC}}}
\newcommand{\cat}{\sf {cat}}
\newcommand{\RP}{{\mathbf {RP}}}
\newcommand{\cd}{{\rm{cd}}}
\newcommand{\vv}{{\mathfrak v}}

\newcommand{\wg}{{\rm {wgt}}}

\begin{document}

\title[Topological complexity]{Motion planning in spaces with small fundamental groups}

\author[Armindo ~Costa and Michael ~Farber]{Armindo ~Costa and Michael ~Farber}
\address{Department of Mathematical Sciences, Durham University\\
South Road, Durham, DH1 3LE, UK}

\email{a.e.costa@durham.ac.uk}
\email{michael.farber@durham.ac.uk}

\maketitle
\begin{abstract}
We establish sharp upper bounds for the topological complexity $\tc(X)$ of motion planning algorithms in topological spaces $X$ such that the fundamental group is {\it \lq\lq small\rq\rq,} i.e. when  $\pi_1(X)$ is cyclic of order $\le 3$ or has cohomological dimension $\le 2$.
\end{abstract}

\maketitle 

\section{Introduction}

Given a mechanical system, a motion planning algorithm is a function which assigns to any pair of states of the system, an initial state and a desired state, a continuous motion of the system starting at the initial state and ending at the desired state. Design of effective motion planning algorithms is one of
the challenges of modern robotics, see \cite{La}.
Motion planning algorithms are applicable in various situations when the system is autonomous and operates in a fully or partially known environment. As a typical example we can mention collision free control of many particles moving in space or along a graph, see \cite{F3}.

The complexity of motion planning algorithms is measured by a numerical invariant $\tc(X)$ which depends on the homotopy type of the configuration space
$X$ of the system \cite{F1}. This invariant is defined as the Schwarz genus (also known under the term \lq\lq sectional category\rq\rq) of the path-space fibration
\begin{eqnarray}\label{fibration}
p: PX \to X\times X.\end{eqnarray}
Here $PX$ is the space of all continuous paths $\gamma: [0,1]\to X$ equipped with the compact-open topology and $p(\gamma)=(\gamma(0), \gamma(1))$ is the map
associating to a path its pair of end points. $\tc(X)$ is the smallest integer $k$ such that $X\times X$ admits an open cover
$U_1\cup U_2\cup \dots\cup U_k=X\times X$ with the property that there exists a continuous section of (\ref{fibration}) $U_i\to PX$ for each $i=1, \dots, k$.
If $X$ is an Euclidean neighborhood retract then $\tc(X)$ can be equivalently characterized as the minimal integer $k$ such that there exists a section $s: X\times X\to PX$ of fibration $p$ with the property that $X\times X$ can be represented as the union of $k$ mutually disjoint locally compact sets
$$X\times X=G_1\cup \dots\cup G_k$$
such that the restriction $s|G_i$ is continuous for $i=1, \dots, k$, see \cite{invit}, Proposition 4.2. A section $s$ as above represents a motion planning algorithm: given a pair $(A,B)\in X\times X$ the image $s(A, B)\in PX$ is a continuous motion of the system starting at the state $A$ and ending at the state $B$.

Intuitively, the topological complexity $\tc(X)$ can be understood as a measure of the navigational complexity of the topological space $X$; it is the minimal number of continuous rules which are needed to describe a motion planning algorithm in $X$.

$\tc(X)$ admits an upper bound in terms of the dimension of the configuration space $X$,
\begin{eqnarray}\label{upperbound}
\tc(X) \le 2\dim(X) +1
\end{eqnarray}
see \cite{F1}, Theorem 4. There are many examples when inequality (\ref{upperbound}) is sharp: take for instance $X= T^n\sharp T^n$, the connected sum of two copies of a torus. However for any simply connected space $X$ one has a more powerful upper bound
\begin{eqnarray}\label{ineq1}\tc(X) \le \dim(X) +1,\end{eqnarray}
see \cite{F2}. The latter inequality is sharp for any simply connected closed symplectic manifold $X$, see \cite{FTY}.
Our goal in this paper is to establish results intermediate between (\ref{upperbound}) and (\ref{ineq1}) under various assumptions on the fundamental group $\pi_1(X)$. We start by stating the following theorem:

\begin{theorem}\label{thm1} Let $X$ be a cell-complex with $\pi_1(X)=\Z_2$. Then
\begin{eqnarray}\label{upperbound1}\tc(X) \le 2\dim (X).\end{eqnarray}
Moreover, for a closed manifold $X$ with  $\pi_1(X)=\Z_2$ one has
\begin{eqnarray}\label{upperbound2}\tc(X) \le 2\dim (X)-1\end{eqnarray}
assuming that $w^n=0$ where $n=\dim(X)$ and $w\in H^1(X;\Z_2)$ is the generator.
\end{theorem}

One knows that $\tc(\RP^n)\le 2n$ for all $n$ (in consistence with (\ref{upperbound1})); moreover, $\tc(\RP^n)=2n$ if and only if $n$ is a power of $2$, see Corollary 14 of \cite{FTY}.

Theorem \ref{thm1} contrasts the related results for the Lusternik - Schnirel\-mann category $\cat(X)$. Recall that $\cat(X)$ denotes the smallest integer $k$ such that
$X$ admits an open cover $X=V_1\cup \dots \cup V_k$ with the property that each inclusion $V_i\to X$ is null-homotopic where $i=1, \dots, k$.
The general dimensional upper bound
$$\cat(X)\le \dim(X) +1$$
 is sharp for all real projective spaces. Theorem 3.5 of Berstein \cite{Bers} states that for a closed connected $n$-dimensional manifold $X$ with $\pi_1(X) =\Z_2$ one has
 $\cat(X)=\dim (X)+1$ if and only if $w^n\not=0\in H^n(X;\Z_2)$ where $w\in H^1(X;\Z_2)$ is the generator.

Theorem \ref{thm1} raises questions about sharp upper bounds for $\tc(X)$ for spaces with other \lq\lq small\rq\rq\, fundamental groups.
The case when $\pi_1(X)=\Z_3$ is addressed by
Theorem \ref{thm2} below; the answer is quite different from Theorem \ref{thm1}:

\begin{theorem}\label{thm2} Let $X$ be a finite cell complex with $\pi_1(X) = \Z_3$. \newline {\rm {(i)}} Assume that either $\dim X$ is odd or $\dim X =2n$ is even and the $3$-adic expansion of $n$ contains at least one digit $2$. Then,
\begin{eqnarray}\label{ineq3}
\tc(X) \le 2\dim(X).
\end{eqnarray}
{\rm {(ii)}} For any integer $n\ge 1$ having only the digits $0$ and $1$ in its $3$-adic expansion there exists a $2n$-dimensional finite polyhedron $X$ with $\pi_1(X)=\Z_3$ and
$$\tc(X) = 2\dim (X)+1.$$
\end{theorem}

There are examples when inequality (\ref{ineq3}) of statement (i) is sharp. In paper \cite{FG} it is shown that for the lens space $X=L_3^{2n+1}$ one has
$\tc(X)=2\dim(X)$ for any $n$ having only digits $0$ and $1$ in its 3-adic expansion. Here $L^{2n+1}_3$ is the factor space $S^{2n+1}/\Z_3$ where $\Z_3=\{1, \omega, \omega^2\}$ is the group of 3-roots of unity, $\omega=e^{2\pi i/3}$.

Next we mention the following result applicable to topological spaces $X$ with $\pi_1(X)$ of cohomological dimension $\le 2$:

\begin{theorem}\label{thm3}
Let $X$ be a finite cell complex such that the cohomological dimension of its fundamental group does not exceed $2$, ${\rm {cd}}(\pi_1(X)) \le 2$. Then one has
\begin{eqnarray}\label{tcdran}
\quad\quad  \tc(X) \le \left\{
\begin{array}{ll} \dim(X) +2\cd(\pi_1(X)),&\mbox{if\  $\dim(X)$ is odd},\\ \\
 \dim(X) +2\cd(\pi_1(X))+1, &\mbox{if \ $\dim(X)$ is even.}
 \end{array}\right.
\end{eqnarray}
\end{theorem}

\begin{proof}
This follows from a recent theorem of Dranishnikov \cite{Dran} who studied related questions concerning Lusternik - Schnirelmann category. The theorem of Dranishnikov \cite{Dran} states that for a cell complex $X$ with fundamental group $\pi_1(X)$ of cohomological dimension not exceeding 2 one has
\begin{eqnarray}\label{dran}
\cat(X) \le \lceil(\dim(X) -1)/{2}\rceil+\cd(\pi_1(X))+1.
\end{eqnarray}
 Inequality (\ref{tcdran}) follows from (\ref{dran}) and from the inequality $$\tc(X) \le 2\cdot \cat(X)-1,$$ see \cite{F1}.
\end{proof}

A conjecture, in the spirit the conjecture concerning $\cat(X)$ made by A. Dranishnikov in \cite{Dran}, states:
{\it There exists a function $F: \N \to\N$ such that for any cell-complex $X$ one has}
$$\tc(X) \le \dim(X) + F( \cd(\pi_1(X))).$$

Finally let us mention the following example. Let $X$ be the $d$-dimensional skeleton of a $\mu$-dimen\-sio\-nal torus $T^\mu$, i.e. $X=(T^\mu)^{(d)}$.
Theorem 5.1 of Cohen and Pruidze \cite{CoPr} gives
 that
$\tc(X)=2d+1$, assuming that $\mu\ge 2d\ge 4$. Hence in general an additional assumption that the fundamental group $\pi_1(X)$ is free abelian cannot help to improve the dimensional upper bound (\ref{upperbound}).

\section{Necessary and sufficient condition for $\tc(X)\le 2\dim (X)$}

Given a  connected cell complex $X$, we define below a local coefficient system $I$ over $X\times X$ and a canonical cohomology class
\begin{eqnarray}\mathfrak v = \vv_X \, \in\,  H^1(X\times X;I).\end{eqnarray}

Denote by $G=\pi_1(X, x_0)$ the fundamental group of $X$ and by $I=\ker (\epsilon ) \, \subset \Z[G]$ the kernel of the augmentation homomorphism
$\epsilon: \Z[G]\to \Z$. An element of $I$ is a finite sum of the form $\sum n_ig_i$ where $n_i\in \Z$, $g_i\in G$, $\sum n_i=0$.
 One can view $I$ and $\Z[G]$ as left $\Z[G\times G]$-modules (i.e. as $\Z[G]$-bimodules) via the action
 \begin{eqnarray}\label{monod}(g, h)\cdot \sum n_ig_i = \sum n_i (gg_ih^{-1}), \quad g,h \in G.\end{eqnarray} According to standard conventions
 (see \cite{Wh}, chap. 6), the left $\Z[G\times G]$-modules $I$ and $\Z[G]$ determine local coefficient systems (denoted by $I$ and $\Z[G]$ correspondingly) over $X\times X$.

Consider a map $f: G\times G\to I$ given by \begin{eqnarray}\label{formula}
f(g,h) =gh^{-1}-1, \quad g,h \in G.\end{eqnarray} It is a crossed homomorphism, i.e. it satisfies the identity
$$f((g,h)(g',h'))=f(g,h) +(g,h)f(g'h')$$ where $g, g', h, h'\in G$. By Theorem 3.3 from chapter 6 of \cite{Wh}, $f$ determines a one-dimensional cohomology class $\vv\in H^1(X\times X;I)$.

\begin{lemma}\label{lm0} The restriction of the class $\vv=\vv_X$ to the diagonal $X\subset X\times X$
vanishes, i.e.
\begin{eqnarray}\label{restriction}\vv_X|X=0\in H^1(X;I|X).\end{eqnarray}
\end{lemma}
\begin{proof}
Indeed, the crossed homomorphism induced by $f$ on the diagonal $G\subset G\times G$ is trivial, $f(g,g)=0$ for all $g\in G$ as follows from (\ref{formula}). \end{proof}

Note that the local system $I|X$ corresponds to the augmentation ideal $I$ viewed with the following left $G$-action $$g\cdot \sum n_ig_i=\sum n_i\cdot(gg_ig^{-1}),$$ where $g,g_i\in G$ and $\sum n_i=0$.

Here is another property of the class $\vv=\vv_X$ which is used later in this paper:
\begin{lemma}\label{lm1} One has $$\vv_X=\beta (1)\in H^1(X\times X;I)$$ where
$$\beta: H^0(X\times X;\Z) \to H^1(X\times X;I)$$
is Bockstein homomorphism corresponding to the exact sequence of left $\Z[G\times G]$-modules
$$0\to I \to \Z[G]\stackrel\epsilon \to \Z\to 0.$$
\end{lemma}
\begin{proof} Let $\tilde X$ denote the universal cover of $X$ and let $\tilde x_0\in \tilde X$ be a lift of the base point $x_0\in X$.
Consider the singular chain complex $S_\ast=S_\ast(\tilde X\times \tilde X)$; we identify $S_0$ with the free abelian group generated by the points of
$\tilde X \times \tilde X$. Recall that $S_\ast$ has a structure of a free left $\Z[G\times G]$-module. Consider a $\Z[G\times G]$-homomorphism $k: S_0(\tilde X\times \tilde X)\to \Z[G]$ associating an element of $G$ with every point of $\tilde X\times \tilde X$ and such that $k(\tilde x_0, \tilde x_0)=1\in G$; then $k(g\tilde x_0, h\tilde x_0)=gh^{-1}$.
The cochain $\epsilon\circ k: S_0\to \Z$ represents the class $1\in H^0(X\times X;\Z)$ and the Bockstein image $\beta(1)\in H^1(X\times X;I)$ is represented by the composition
$$\delta(k): \, S_1\stackrel\partial \to S_0\stackrel k\to \Z[G],$$
taking values in $I$. A crossed homomorphism $f': G\times G\to I$ associated to $\beta(1)$ can be found as follows, see
\cite{Wh}, chapter 6, \S 3. Given a pair $(g,h)\in G\times G=\pi_1(X\times X, (x_0,x_0))$, realize it by a loop $\sigma: ([0,1], \partial [0,1]) \to (X\times X, (x_0,x_0))$, then lift $\sigma$ to the covering
$\tilde \sigma: ([0,1], 0) \to (\tilde X\times \tilde X, (\tilde x_0,\tilde x_0))$ and finally apply the cocycle $\delta(k)$ to $\tilde \sigma$, viewed as a singular 1-simplex in $\tilde X\times \tilde X$. We obtain $f'(g,h)=k(g\tilde x_0, h\tilde x_0)- k(\tilde x_0, \tilde x_0) =gh^{-1}-1$ for all $g,h\in G$. This coincides with the crossed homomorphism describing $\vv_X$, see (\ref{formula}).
Hence $\beta(1)=\vv_X$.
\end{proof}
\begin{corollary}\label{cor1}
The order of the class $\vv_X\in H^1(X\times X;I)$ equals the cardinality of the fundamental group $|G|$ of $X$. In particular $\vv_X=0$ if and only if $X$ is simply connected.
\end{corollary}
\begin{proof} Consider the exact sequence
$$H^0(X\times X;I) \to H^0(X\times X;\Z[G]) \stackrel \epsilon \to H^0(X\times X;\Z) \stackrel \beta\to H^1(X\times X;I).$$
Note that $H^0(X\times X;\Z[G])$ is isomorphic to the set of elements $a=\sum n_i g_i\in \Z[G]$ which are invariant with respect to
$G\times G$-action, see \cite{Wh}, chapter 6, Theorem 3.2.
Consider first the case when $G$ is infinite. Then $H^0(X\times X;\Z[G])=0$ as there are no invariant elements in the group ring. Since $H^0(X\times X;\Z)=\Z$ this implies that in this case the class $\vv_X\in H^1(X\times X;I)$ generates an infinite cyclic subgroup.

If $G$ is finite then $H^0(X\times X;I)=0$ (as above) and any $G\times G$-invariant element of $\Z[G]$ is a multiple of $N=\sum_{g\in G} g$.
Hence the group $H^0(X\times X;\Z[G])$ is infinite cyclic generated by $N$ and since
$\epsilon(N) =|G|$,
the exact sequence
$$0\to H^0(X\times X;\Z[G]) \stackrel \epsilon \to H^0(X\times X;\Z) \stackrel \beta\to H^1(X\times X;I)$$
turns into $$0\to \Z \stackrel {|G|}\to \Z \stackrel \beta\to H^1(X\times X;I).$$
This shows that the subgroup of $H^1(X\times X;I)$ generated by the class $\vv_X$ is cyclic of order $|G|$.
\end{proof}

The following result explains the key role the cohomology class $\vv=\vv_X$ plays in the theory of topological complexity.

\begin{theorem}\label{thm4} Let $X$ be a cell complex of dimension $n=\dim (X)\ge 2$.
One has \begin{eqnarray}\tc(X)\le 2n\end{eqnarray} if and only if the $2n$-th power $$\vv^{2n}\, =0 \, \in \, H^{2n}(X\times X; I^{2n})$$ vanishes.
Here $I^{2n}=I\otimes_\Z I\otimes_\Z \dots\otimes_\Z I$ denotes the tensor product over $\Z$ of $2n$ copies of $I$, viewed with the diagonal action of $G\times G$, and $\vv^{2n}$ is the cup-product
$\vv\cup \vv\cup \dots \cup \vv$ of $2n$ copies of $\vv$.
\end{theorem}
\begin{proof}
Consider the path space fibration (\ref{fibration}). The topological complexity $\tc(X)$ is defined as the Schwarz genus of this fibration. Consider also the $2n$-fold fiberwise join  $p_{2n}: P_{2n}X \to X\times X$ of fibration (\ref{fibration}); this construction is described in detail in \cite{Schw}, chapter 2, \S 1.
According to Theorem 3 of Schwarz \cite{Schw}, one has $\tc(X)\le 2n$ if and only the fibration $p_{2n}$ has a continuous section. The fibre $F_{2n}$ of $p_{2n}$ is the $2n$-fold join $\Omega X \ast \Omega X\ast \dots \ast \Omega X$, where $\Omega X$ is the space of based loops in $X$. It follows that the fibre $F_{2n}$ is $(2n-2)$-connected and hence the primary obstruction
\begin{eqnarray*}
\theta_{2n} \in H^{2n}(X\times X; \mathcal L^{2n}),  \quad \mathcal L^{2n} =\{ \pi_{2n-1}(F_{2n})\}=\{ H_{2n-1}(F_{2n})\},\end{eqnarray*}  is the only obstruction to the existence of a continuous section of $p_{2n}$:
one has $\tc(X)\le 2n$ if and only if $\theta_{2n}=0$.
Note that  the fibre $F_{2n}$ is 2-connected since $n\ge 2$.
The symbol
$\mathcal L^{2n} = \{ H_{2n-1}(F_{2n})\}$ denotes a local system of homology groups of fibres which associates with any point $(x, y)\in X\times X$ of the base the abelian group $$\mathcal L^{2n}_{(x,y)}=H_{2n-1}(p_{2n}^{-1}(x, y))$$ and with any path $\sigma: [0,1]\to X\times X$ an isomorphism
\begin{eqnarray}\label{iso}
\sigma_\ast: \mathcal L^{2n}_{\sigma(1)} \to \mathcal L^{2n}_{\sigma(0)}\end{eqnarray} defined as follows. Given $\sigma$ one applies the Homotopy Lifting Property to find a map
\begin{eqnarray}\label{lift}K: p_{2n}^{-1}(\sigma(1))\times [0,1] \to P_{2n}X\end{eqnarray}
satisfying $K(a,1)=a$ and $p_{2n}(K(a,t))=\sigma(t)$ for all $a\in p_{2n}^{-1}(\sigma(1))$ and $t\in [0,1]$. Then $a\mapsto K(a,0)$ is a map
$p_{2n}^{-1}(\sigma(1)) \to p_{2n}^{-1}(\sigma(0))$ and (\ref{iso}) is the induced map on homology.

By Theorem 1 from \cite{Schw} the local system $\mathcal L^{2n}$ is the tensor power of $2n$ copies of a local system $\mathcal L$,
$$\mathcal L^{2n} = \mathcal L \otimes \mathcal L \otimes \dots \otimes \mathcal L, \quad  \, \mbox{($2n$ times)}$$ and the obstruction $\theta_{2n}$ is
$2n$-fold cup-product
$$\theta_{2n}=\theta\cup \theta \cup \dots\cup \theta\quad \, \mbox{($2n$ times)}, \quad \theta\in H^1(X\times X;\mathcal L).$$
Here $\mathcal L$ is the local system $\mathcal L_{(x,y)}=\tilde H_0(p^{-1}(x,y))$ of reduced zero dimensional homology groups of fibres of the initial fibration (\ref{fibration}) and $\theta\in H^1(X\times X;\mathcal L)$ is \lq\lq the homological obstruction\rq\rq\, to the existence of a continuous section of (\ref{fibration}) over the 1-skeleton of $X\times X$. Theorem \ref{thm4} follows once we are able to identify the local systems $\mathcal L$ and $I$ so that $\theta =\vv \in H^1(X\times X;I)$.

Let $x_0\in X$ the base point. The fiber $p^{-1}(x_0, x_0)$ is the space $\Omega X$ of all loops in $X$ based at $x_0$. Path-connected components of the fibre are in one-to-one correspondence with elements of of the fundamental group $\pi_1(X, x_0)=G$. We see that $H_0(p^{-1}(x_0, x_0))=\Z[G]$ and $\mathcal L_{(x_0,x_0)}= \tilde H_0(p^{-1}(x_0, x_0))=I= \ker(\epsilon)$. Next we show that the monodromy on $I$ acts according to (\ref{monod}). Given a path $\sigma: [0,1]\to X\times X$, where $\sigma(t)=(\alpha(t), \beta(t))$,
with $\sigma(0)=\sigma(1)=(x_0,x_0)$, we may define a homotopy $$K_\sigma: p^{-1}(x_0, x_0)\times [0,1] \to PX$$ similar to (\ref{lift}) by the formula:
$$K_\sigma (\omega, \tau)(t) = \left\{ \begin{array}{ll}
\alpha(3t+\tau), & \mbox{for}\,\, 0\le t\le \frac{1 -\tau}{3},\\ \\
\omega(\frac{3t+\tau -1}{1+2\tau}), & \mbox{for}\, \, \frac{1 -\tau}{3}\le t \le \frac{2+\tau}{3},\\ \\
\beta(-3t+\tau +3), & \mbox{for}\, \, \frac{2+\tau}{3}\le t\le 1,
\end{array}\right.
$$
where $\omega \in p^{-1}(x_0, x_0)=\Omega X$ and $t, \tau\in [0,1]$. One has $K(\omega, 1)=\omega$ and $p(K(\omega, \tau))=\sigma(\tau)$.
The monodromy action $\Omega X \to \Omega X$ along $\sigma$ is given by
\begin{eqnarray}\label{monod1}\omega\mapsto K_\sigma(\omega, 0)= \alpha \omega \bar{\beta};
 \end{eqnarray}
 here $\bar{\beta}$ denotes the inverse loop to $\beta$. We see that this map induces on $\tilde H_0(\Omega X)$ the monodromy action (\ref{monod}) and therefore $I$ and $\mathcal L$ coincide as local coefficient systems.

Finally we show that the homological obstruction $\theta\in H^1(X\times X;I)$ equals $\vv$. Without loss of generality we may assume that
$X$ has a single zero-dimensional cell $x_0$. Let $\omega_0$ be the constant loop at $x_0$; this defines a section over the 0-skeleton.
The homological obstruction associates with any oriented 1-cell of $X\times X$ the formal difference, in $\tilde H_0(\Omega X)=I$, between the connected components of $K_\sigma(\omega_0)$ and $\omega_0$ where $\sigma$ is a loop representing the cell.
For any oriented 1-cell $e$ of $X$ consider the corresponding oriented one-cells $e\times x_0$ and $x_0\times e$ of $X\times X$.
As follows from formula (\ref{monod1}) the crossed homomorphism $f': G\times G\to I$ corresponding to $\theta$ is given by
$$f'(g, 1) = g-1, \quad f'(1, h) = h^{-1}-1, \quad h\in G.$$
Hence, we see that $$f'(g,h)= f'((g,1)(1,h))= f'(g,1) +(g,1)f'(1,h)=gh^{-1}-1 =f(g,h).$$
Therefore $\theta=\vv$.
\end{proof}

\begin{corollary}
Let $X$ be a cell complex with $\tc(X)=2\dim(X)+1.$ Then the topological complexity of the Eilenberg - MacLane complex $Y=K(\pi_1(X),1)$ satisfies $$\tc(Y) \ge 2\dim(X)+1.$$\end{corollary}
\begin{proof} $X$ is aspherical if $\dim X=1$. Hence we may assume that $n=\dim(X)\ge 2$ and so Theorem \ref{thm4} is applicable.
Consider local systems $I_X$ on $X\times X$ and $I_Y$ on $Y\times Y$ and cohomology classes $\vv_X\in H^1(X\times X;I_X)$ and
$\vv_Y\in H^1(Y\times Y;I_Y)$ as described above. The canonical map $f: X\to Y$ inducing an isomorphism of fundamental groups satisfies $(f\times f)^\ast(I_Y)=I_X$ and $(f\times f)^\ast(\vv_Y)=\vv_X$. If $(\vv_X)^{2n}\not=0$ then $(\vv_Y)^{2n}\not=0$. Inequality $\tc(Y)\ge 2n+1$ now follows from \cite{invit}, Corollary 4.40 since $\vv_Y$ is a zero-divisor.
\end{proof}

\section{Proof of Theorem \ref{thm1}}

Let $X$ be a connected cell complex with $\pi_1(X)=\Z_2=G$. Clearly, $n=\dim(X)\ge 2$ and we may apply Theorem \ref{thm4}. The augmentation ideal
$I=\ker[\epsilon: \Z[G]\to \Z]$ is isomorphic to $\Z$ as an abelian group; however $I$  is nontrivial as a local system on $X\times X$. More precisely, each of the classes $(g,1), (1,g)\in G\times G$ (where $g\in G$ is the unique nontrivial element) acts as multiplication by $-1$ on $\Z=I$.
It follows that the tensor square $I\otimes_\Z I$ is the trivial coefficient system $\Z$.

Consider the canonical class $\vv=\vv_X\in H^1(X\times X;I)$ and its square $$\vv^2\in H^2(X\times X;\Z).$$ Since $H^1(X;\Z)=0$ the K\"unneth theorem gives
$$H^2(X\times X)=H^2(X)\otimes H^0(X) \oplus H^0(X)\otimes H^2(X)$$
where we dropped the coefficient group $\Z$ from the notation. Hence we may write
$$\vv^2=a\times 1 + 1\times b, \quad a, b \in H^2(X; \Z).$$
By Lemma \ref{lm0} one has $a+b=0$, and by Corollary \ref{cor1} both classes $a$ and $b$ are of order two: $2a=0=2b$. Hence we may write
$$\vv^2 =a\times 1+1\times a$$
and
$$\vv^{2n} =(\vv^2)^n =(a\times 1+1\times a)^n = \sum_{i=0}^n \binom{n}{i} a^i\times a^{n-i}.$$
If $n$ is odd then in the last sum either $a^i=0$ or $a^{n-i}=0$ for dimensional reasons. If $n$ is even then
$$\vv^{2n} = \binom{n}{n/2} a^{n/2}\times a^{n/2}=0$$
since the binomial coefficient $\binom{n}{n/2}$ is always even and $2a=0$. Theorem \ref{thm4} implies now that $\tc(X)\le 2n$.

To prove the second statement of the Theorem, assume that $X$ is a closed manifold satisfying $\pi_1(X)=\Z_2$ and $w^n=0$ where $w\in H^1(X;\Z_2)$ is the generator. By a Theorem of Berstein \cite{Bers} mentioned earlier one has $\cat(X)\le \dim(X)$. Our statement (\ref{upperbound2}) follows now from the inequality
$\tc(X)\le 2\cat(X)-1$, see \cite{F1}.
This completes the proof.

\section{Proof of Theorem \ref{thm2}}

Let $X$ be a cell complex such that the fundamental group $\pi_1(X, x_0)=G=\{1, t, t^2\}$ is cyclic of order $3$, i.e. $t^3=1$. The group ring $\Z[G]$
is the ring of polynomials of the form $a+bt+ct^2$ with the usual addition and multiplication and with the additional identity $t^3=1$; here the coefficients $a, b,c\in \Z$ are integers. The augmentation ideal $I$ has rank $2$; it
is generated over $\Z$ by two elements $\alpha=t-1$ and $\beta=t^2-t$. The structure of $I$ as a $\Z[G\times G]$-module is given by
$$(t,1)\cdot \alpha=\beta, \quad  (t,1)\cdot \beta= -\alpha - \beta,$$
and
$$(1,t)\cdot \alpha=- \alpha -\beta, \quad (1,t)\cdot \beta= \alpha.$$

Consider the canonical class $\vv_X\in H^1(X\times X; I)$ and its square $\vv_X^2\in H^2(X\times X; I\otimes I)$. The local system $I\otimes I$ has rank 4
and is generated by the elements $\alpha\otimes \alpha$, $\alpha\otimes \beta$, $\beta\otimes \alpha$ and $\beta\otimes \beta$. The $\Z[G\times G]$-action is diagonal, for example
$$(t,1)\cdot \alpha\otimes \alpha = \beta\otimes\beta,$$
$$(t,1)\cdot \alpha\otimes \beta= \beta\otimes (-\alpha -\beta)= -\beta\otimes \alpha - \beta\otimes \beta$$
and so on.
Consider the homomorphism
$$T: I\otimes I\to I\otimes I$$
which interchanges the factors. One has $T(\alpha\otimes \beta)=\beta\otimes \alpha$, $T(\beta\otimes \alpha)=\alpha\otimes\beta$ and $T$ acts identically on
two other generators $\alpha\otimes \alpha$ and $\beta\otimes\beta$. It is easy to see that $T$ is a $\Z[G\times G]$-homomorphism and hence can be viewed as a homomorphism of local systems.

Let $I\wedge I\subset I\otimes I$ denote the subgroup generated by the element $\alpha\otimes \beta - \beta\otimes \alpha$. One observes that
$I\wedge I=\Z$ has a trivial $\Z[G\times G]$- action; in particular it is a $\Z[G\times G]$-submodule of $I\otimes I$. Denote the factor module by $S(I)$; it is the symmetric square of $I$. We have the following exact sequence of local systems
$$0\to I\wedge I  \stackrel i \to I\otimes I \stackrel j \to S(I) \to 0$$
(recall that $I\wedge I=\Z$ is trivial) which induces an exact sequence
$$\to H^n(X\times X;I\wedge I) \stackrel{i_\ast}\to H^n(X\times X; I\otimes I) \stackrel{j_\ast}\to H^n(X\times X; S(I))\to $$

We claim that the class $\vv_X^2\in H^2(X\times X;I\otimes I)$ satisfies
\begin{eqnarray}j_\ast(\vv_X^2)\, =\, 0 \, \in \, H^2(X\times X; S(I)).\end{eqnarray}
From the skew-commutativity property of cup-products it follows that $T_\ast(\vv_X^2)= -\vv_X^2$. Since $j=j\circ T$ we obtain
$$j_\ast(\vv^2_X) = j_\ast T_\ast(\vv_X^2) = - j_\ast(\vv_X^2),\quad \mbox{i.e.} \quad 2j_\ast(\vv_X^2) =0.$$
On the other hand, by Corollary \ref{cor1} one has $3j_\ast(\vv_X^2)=0$ which together with above implies that $j_\ast(\vv_X^2)=0$.

From the long exact cohomological sequence sequence we obtain that
\begin{eqnarray}\label{w}
\vv_X^2 = i_\ast(w)\quad \mbox{for some}\quad w\in H^2(X\times X;\Z).
\end{eqnarray}
We claim that any class $w\in H^2(X\times X;\Z)$ satisfying (\ref{w}) is annihilated by multiplication by $6$, i.e.
\begin{eqnarray}\label{six}
6w=0.
\end{eqnarray}
Indeed, consider the map $A: I\otimes I\to I \wedge I =\Z$ given by $A(x)=x-T(x)$ for $x\in I\otimes I$. Clearly $A$ is a homomorphism of local systems and $A\circ i: I\wedge I \to I\wedge I$ is multiplication by $2$. Hence we obtain $2w=A_\ast\circ i_\ast(w)=A_\ast(\vv_X^2)$ which implies (\ref{six}) since $3\vv_X=0$.

By the K\"unneth theorem using $H^1(X;\Z)=0$ one can write
$$w=a\times 1 +1\times b$$ where $a, b\in H^2(X;\Z)$ with $6a=0=6b$.
Then $$\vv_X^{2n}= (\vv_X^2)^n= i_\ast(w^n) = \sum_{k=0}^n \binom{n}{k}\, i_\ast(a^k\times b^{n-k}).$$
If $n$ is odd each term in the last sum vanishes for dimensional reasons. Suppose now that $n$ is even, $n=2m$. Then we obtain
$$\vv_X^{2n} = \binom{2m}{m} a^{m}\times b^{m}.$$
We have already mentioned that the binomial coefficient $\binom{2m}{m}$ is always even. It is divisible by $3$ if the 3-adic expansion of $m$ contains at least one digit $2$, see \cite{FG}, Lemma 19. This shows that $\vv_X^{2n}=0$ under the conditions indicated in statement (i) of Theorem \ref{thm2} and implies statement (i) by applying Theorem \ref{thm4}.

Next we prove statement (ii) of Theorem \ref{thm4}. Let $n\ge 1$ be such that its 3-adic expansion contains only digits 0 and 1. Then the binomial coefficient
$\binom{2n}{n}$ is not divisible by $3$, see \cite{FG}, Lemma 19.

Consider the lens space $L_3^{2n+1}=S^{2n+1}/\Z_3$ where $S^{2n+1}\subset \C^{n+1}$ is the unit sphere and $\Z_3=\{1, \omega, \omega^2\}$ acts as the group of roots of $1$, where $\omega=\exp{2\pi i/3}$. It is well known that the lens space has a cell decomposition with a unique cell in every dimension $i$ for $i=0, 1, \dots, 2n+1$, see \cite{Hat}, page 144-145. We will denote by $X$ the skeleton of $L_3^{2n+1}$ of dimension $2n$. Note that $X$ has homotopy type of the lens space $L_3^{2n+1}$ with one point removed. We show below that $\tc(X)=4n+1$ using the technique developed in \cite{FG}.

The cohomology algebra $H^\ast(X;\Z_3)$ can be described as the quotient of the polynomial algebra $\Z_3[x, y]$ with two generators $x$ of degree 1 and $y$ of degree 2 subject to relations $x^2=0$, $y^{n+1}=0$ and $xy^n=0$, see \cite{Hat}, page 251. Here $x\in H^1(X;\Z_3)$ is the generator and $$y=\beta(x)\in H^2(X;\Z_3)$$ is the image of $x$ under the Bockstein homomorphism $\beta: H^1(X; \Z_3)\to H^2(X;\Z_3)$ corresponding to the exact sequence
$$0\to \Z_3\to \Z_9\to \Z_3\to 0.$$
The classes $y^k$, where $k=0,1, \dots, n$, together with $xy^j$ for all $j=0, 1, \dots, n-1$ form an additive basis of $H^\ast(X;\Z_3)$.
By the K\"unneth theorem one has
$$H^\ast(X\times X;\Z_3) = H^\ast(X; \Z_3) \otimes H^\ast(X;\Z_3)$$ and therefore the classes
$$x^ay^b\times x^cy^d\, \in\,  H^\ast (X\times X;\Z_3)$$
 where $a, c\in \{0,1\}$ and $b, d\in \{0,1, \dots,n\}$ and
$(a, b)\not=(1, n)$, $(c,d)\not=(1, n)$ form an additive basis.
We denote by $\bar x$ and
$\bar y$ the classes $$\bar x= x\times 1-1\times x\in H^1(X\times X;\Z_3), \quad \bar y= y\times 1-1\times y\in H^2(X\times X;\Z_3).$$
It is shown in \cite{FG} that $\beta(\bar x) = \bar y$ and therefore the class $\bar y$ has weight two with respect to fibration (\ref{fibration}).

Recall that a cohomology class $u\in H^\ast(X\times X;R)$ is said to have {\it weight greater than or equal to $k$} (notation $\wg(u) \ge k$) if the restriction $u|A=0$ vanishes for any open subset
$A\subset X\times X$ with $\tc_X(A)\le k$, see \cite{FG} and \cite{invit}, \S 4.5. Here $\tc_X(A)$ denotes {\it the relative topological complexity} of a subset
$A\subset X\times X$; the latter is defined as the smallest number $r$ such that $A$ admits an open cover $A=U_1 \cup \dots\cup U_r$ with the property that the projections
$X\leftarrow U_i \to X$ on the first and the second factors are homotopic to each other, for all $i=1, \dots, r$.

 By Lemma 4.39 of \cite{invit} one has $$\wg((\bar y)^{2n})\ge 2n \cdot \wg(\bar y)  \ge 4n$$ and the nontriviality of the power $(\bar y)^{2n}\in H^{4n}(X\times X;\Z_3)$ would imply
 $\tc(X) \ge 4n+1$, according to Proposition 4.36 of \cite{invit}. The opposite inequality $\tc(X)\le 4n+1$ follows directly from (\ref{upperbound}) giving
 $\tc(X) =4n+1$ as desired.

By a direct computation one has
 $$(\bar y)^{2n} = (-1)^n \binom{2n}{n} y^n\times y^n$$
 and the binomial coefficient $\binom{2n}{n}$ is mutually prime to $3$ due to the fact that the 3-adic expansion of $n$ involves only small digits $n_i\in \{0,1\}$, see Appendix B from \cite{FG}. Thus we obtain $(\bar y)^{2n} \not = 0$ completing the proof.

\bibliographystyle{alpha}

\end{document}